\documentclass[dvipdfmx]{article}
\usepackage{authblk}
\usepackage{a4wide}
\usepackage{amsthm}
\usepackage{graphicx}
\usepackage{amssymb}
\usepackage{amsmath}
\usepackage{ascmac}
\usepackage{setspace}
\usepackage{float}
\usepackage{algorithm}
\usepackage{algorithmic}
\usepackage{setspace}
\usepackage{tikz-cd}
\usepackage{comment}
\newtheorem{Theorem}[equation]{Theorem}

\newtheorem{Proposition}[equation]{Proposition}

\newtheorem{Definition}[equation]{Definition}

\theoremstyle{remark}
\newtheorem{Remark}[equation]{Remark}
\numberwithin{equation}{section}

\DeclareMathOperator{\End}{End}

\DeclareMathOperator{\id}{id}

\DeclareMathOperator{\ad}{ad}

\DeclareMathOperator{\row}{row}
\DeclareMathOperator{\col}{col}

\newcommand{\ve}{\varepsilon}

\allowdisplaybreaks
\begin{document}
\title{Notes on a homomorphism from the affine Yangian associated with $\widehat{\mathfrak{sl}}(n)$ to the affine Yangian associated with $\widehat{\mathfrak{sl}}(n+1)$}
\author{Mamoru Ueda}
\date{}
\maketitle
\begin{abstract}
In the previous paper \cite{U8}, we constructed a homomorphism from the affine Yangian associated with $\widehat{\mathfrak{sl}}(n)$ to the standard degreewise completion of the affine Yangian associated with $\widehat{\mathfrak{sl}}(n+1)$. In this article, by using this homomorphism, we construct homomorphisms from the affine Yangian to the universal enveloping algebra of a non-rectangular $W$-algebra of type $A$, which are different from the one in \cite{U7}.  Based on this result, we discuss a new definition of the affine shifted Yangian, from which it is expected that there exists a homomorphism to the universal enveloping algebra of a non-rectangular $W$-algebra of type $A$. In the appendix, we also show that this homomorphism is related to the one from the quantum affine algebra associated with $\widehat{\mathfrak{sl}}(n)$ to the quantum affine algebra associated with $\widehat{\mathfrak{sl}}(n+1)$ given by Li \cite{Li}.
\\
\qquad{\bf Affiliation;} Graduate School of Mathematical Sciences, The University of Tokyo\\ 
\qquad{\bf Address;} 3-8-1 Komaba, Meguro-ku, Tokyo 153-8914, Japan\\
\qquad{\bf E-mail address;} mueda@ms.u-tokyo.ac.jp\\
\end{abstract}

\section{Introduction}

For $\hbar\in\mathbb{C}$, Drinfeld (\cite{D1}, \cite{D2}) introduced the Yangian $Y_\hbar(\mathfrak{g})$ associated with a finite dimensional simple Lie algebra $\mathfrak{g}$. The Yangian $Y_\hbar(\mathfrak{g})$ is a quantum group which is a deformation of the current algebra $\mathfrak{g}\otimes\mathbb{C}[z]$. For $\hbar,\ve\in\mathbb{C}$, Guay \cite{Gu2} defined the 2-parameter affine Yangian $Y_{\hbar,\ve}(\widehat{\mathfrak{sl}}(n))$, which is a deformation of the universal enveloping algebra of the central extension of $\mathfrak{sl}(n)\otimes\mathbb{C}[u^{\pm1},v]$. In the previous paper \cite{U8}, the author constructed a homomorphism, for $n<m$,

\begin{equation*}
\Psi^{n,m}\colon Y_{\hbar,\ve}(\widehat{\mathfrak{sl}}(n))\to \widetilde{Y}_{\hbar,\ve}(\widehat{\mathfrak{sl}}(m)),
\end{equation*}
where $\widetilde{Y}_{\hbar,\ve}(\widehat{\mathfrak{sl}}(n))$ is the standard degreewise completion of the affine Yangian $Y_{\hbar,\ve}(\widehat{\mathfrak{sl}}(n))$. In this article, we consider a relationship between the affine Yangian $Y_{\hbar,\ve}(\widehat{\mathfrak{sl}}(n))$ and a $W$-algebra of type $A$ by using $\Psi^{n,m}$. 

A $W$-algebra $\mathcal{W}^k(\mathfrak{g},f)$ associated with a finite dimensional reductive Lie algebra $\mathfrak{g}$, a nilpotent element $f\in\mathfrak{g}$ and a complex number $k\in\mathbb{C}$ is a vertex algebra defined by the quantum Drinfeld-Sokolov reduction.
It appeared in the study of two dimensional conformal field theories (\cite{Z}).
We call a $W$-algebra associated with $\mathfrak{gl}(n)$ (resp. $\mathfrak{gl}(ln)$) and a principal nilpotent element (resp. a nilpotent element of type $(l^n)$) a principal (resp. rectangular) $W$-algebra of type $A$. 

Recently, the relationship between a $W$-algebra and the Yangian has been studied by mathematicians and physicists. 
Schiffmann and Vasserot \cite{SV} gave the representation of the principal $W$-algebra of type $A$ on the equivariant homology space of the moduli space of $U(r)$-instantons by using an action of the affine Yangian associated with $\widehat{\mathfrak{gl}}(1)$ on this homology space. From the perspective of physics, Schiffmann-Vasserot's \cite{SV} results correspond to the solution of the AGT (Alday-Gaiotto-Tachikawa) conjecture. 
As for rectangular cases, the author \cite{U4} constructed a surjective homomorphism from the affine Yangian to the universal enveloping algebra of a rectangular $W$-algebra of type $A$.

In non-rectangular cases, Creutzig-Diaconescu-Ma \cite{CE} conjectured that a geometric representation of an iterated $W$-algebra of type $A$ on the equivariant homology space of the affine Laumon space will be given through an action of the shifted affine Yangian constructed in \cite{FT}. Creutzig-Diaconescu-Ma's conjecture corresponds to the generalized AGT conjecture. Based on this conjecture, it is expected that there exists a non-trivial homomorphism from the shifted affine Yangian to the universal enveloping algebra of an iterated $W$-algebra associated with any nilpotent element.
However, the defining relations of the shifted affine Yangian are so complicated that it is difficult to solve this homomorphism directly. 

In finite setting, the relationship between the finite Yangian and a finite $W$-algebra has been studied. A finite $W$-algebra $\mathcal{W}^{\text{fin}}(\mathfrak{g},f)$ is an associative algebra associated with a finite dimensional reductive Lie algebra $\mathfrak{g}$ and a nilpotent element $f$. Brundan-Kleshchev \cite{BK} wrote down the finite $W$-algebra of type $A$ as a quotient algebra of the shifted Yangian of type $A$, a subalgebra of the finite Yangian of type $A$ associated with anti-dominant weight. The shifted Yangian contain finite Yangians of type $A$. By restricting the homomorphism of Brundan-Kleshchev \cite{BK} to these finite Yangians, we obtain homomorphisms from a finite Yangian of type $A$ to the finite $W$-algebra of type $A$. De Sole-Kac-Valeri \cite{DKV} gave one of these homomorphisms by using the Lax operator.

In Section~4, we give the affine version of this homomorphism.
Let us take a positive integer $s$, integers $q_1>q_2>q_3>\cdots>q_{s-1}>q_s>0(=q_{s+1})$ and positive integers $\{l_u\}_{u=1}^s$ and fix a nilpotent element $f_{q_1,q_2,\cdots,q_s}^{l_1,l_2\cdots,l_s}\in\mathfrak{gl}(\sum_{v=1}^s\limits q_vl_v)$ which is of type 
\begin{equation*}
(l_1^{q_1-q_2},(l_1+l_2)^{q_2-q_3},\cdots,(\sum_{v=1}^{s-1}l_v)^{q_{s-1}-q_{s}},(\sum_{v=1}^{s}l_v)^{q_s}).
\end{equation*} 
\begin{Theorem}\label{A}
For $1\leq u\leq s$, we suppose that $q_u-q_{u+1}\geq3$ and $\dfrac{\ve+q_u\hbar}{\hbar}=k+\sum_{v=1}^s\limits l_vq_v$. Then, we can obtain the homomorphism
\begin{equation*}
\Phi_u\colon Y_{\hbar,\ve}(\widehat{\mathfrak{sl}}(q_u-q_{u+1}))\to \mathcal{U}(\mathcal{W}^k(\mathfrak{gl}(\sum_{v=1}^s\limits l_vq_v),f_{q_1,q_2,\cdots,q_s}^{l_1,l_2\cdots,l_s})),
\end{equation*}
where $\mathcal{U}(\mathcal{W}^k(\mathfrak{gl}(\sum_{v=1}^s\limits l_vq_v),f_{q_1,q_2,\cdots,q_s}^{l_1,l_2\cdots,l_s}))$ means the universal enveloping algebra of a $W$-algebra \\$\mathcal{W}^k(\mathfrak{gl}(\sum_{v=1}^s\limits l_vq_v),f_{q_1,q_2,\cdots,q_s}^{l_1,l_2\cdots,l_s})$.
\end{Theorem}
In \cite{U7}, we have already given the homomorphism $\Phi_u$. In this article, we construct other homomorphisms by using $\Psi^{n,m}$ and the proof of \cite{U7}. In \cite{U12}, the author has given another proof of Theorem~\ref{A} by using a homomorphism
\begin{equation*}
\widetilde{\Psi}^{n,m}\colon Y_{\hbar,\ve+(m-n)\hbar}(\widehat{\mathfrak{sl}}(n))\to \widetilde{Y}_{\hbar,\ve}(\widehat{\mathfrak{sl}}(m)),
\end{equation*}
which is a different homomorphism from $\Psi^{n,m}$. The proof in this article does not need $\widetilde{\Psi}^{m-n,m}$ and is shorter than the one in \cite{U12}.
Similarly to the finite setting, it is expected that we can obtain a homomorphism from the shifted affine Yangian to the universal enveloping algebra of the iterated $W$-algebra by extending homomorphisms $\Phi_u$.

In section~5, we will show how the homomorphism $\Phi_u$ induces a restriction of the homomorphism in \cite{BK} to finite Yangians. In Section 6, we also note how $\Psi^{n,n+1}$ is related to the map
\begin{equation*}
\Psi^{n,n+1}_f\colon Y_\hbar(\mathfrak{sl}(n))\to Y_\hbar(\mathfrak{sl}(n+1))
\end{equation*}
given by $H_{i,r}\mapsto H_{i,r}$ and $X^\pm_{i,r}\mapsto X^\pm_{i,r}$ in the current presentation of the finite Yangian $Y_\hbar(\mathfrak{sl}(n))$.
Li \cite{Li} constructed an embedding from the quantum affine algebra associated with $\widehat{\mathfrak{sl}}(n)$ to the quantum affine algebra associated with $\widehat{\mathfrak{sl}}(n+1)$. In \cite{Li2}, Li gave the similar embedding for the quantum group associated with a general symmetrizable Kac-Moody Lie algebra. In the appendix, we rewrite the homomorphism in \cite{Li} and relate it to the homomorphism $\Psi^{n,n+1}_f$.
In \cite{Li3}, Li also showed that the embedding is compatible with the Schur-Weyl duality for quantum affine algebra. We give the relationship between the homomorphism $\Psi^{n,n+1}_f$ and the Schur-Weyl duality. 

There exist another version of Creutzig-Diaconescu-Ma's conjecture, which notes that there exists a surjective homomorphism from the shifted affine Yangian to the universal enveloping algebra of a non-rectangular $W$-algebra of type $A$ for a properly defined shifted affine Yangian. Let us consider the case that $s=2$. Naturally, it is expected that we can extend homomorphisms $\Phi_1$ and $\Phi_2$ directly to the subalgebra of $Y_{\hbar,\ve}(\widehat{\mathfrak{sl}}(q_1))$ generated by $\Psi^{q_1-q_2,q_1}(Y_{\hbar,\ve}(\widehat{\mathfrak{sl}}(q_1-q_2)))$, $\widetilde{\Psi}^{q_2,q_1}(Y_{\hbar,\ve+(q_1-q_2)\hbar}(\widehat{\mathfrak{sl}}(q_2)))$, $X^+_{m-n,0}$ and $\{E_{m-n,m-n}t^{y}-E_{m-n+1,m-n+1}t^y\mid y\in\mathbb{Z}\}$ and this extension becomes a restriction of the homomorphism of Creutzig-Diaconescu-Ma's conjecture. Unfortunately, this expectation is not correct.
In Section~7, we show that some relations of the generators of these algebras are not satisfied by using OPEs computed in Section~3. 
We expect that these results of computation will lead to the new definition of the shifted affine Yangian.

\section*{Acknowledgement}
The author wishes to express his sincere gratitude to Yiqiang Li. This paper is motivated by his comment on \cite{U8}. The author is also grateful to Nicolas Guay and Thomas Creutzig for the helpful discussion. 

\section{Affine Yangian}
In this section, we recall the definition of the affine Yangian of type $A$. We identify $\{0,1,2,\cdots,n-1\}$ with $\mathbb{Z}/n\mathbb{Z}$. Let us set
\begin{align*}
a_{i,j}&=\begin{cases}
2&\text{if } i=j, \\
-1&\text{if }j=i\pm 1,\\
0&\text{otherwise.}
	\end{cases}
\end{align*}
\begin{Definition}[Definition 3.2 in \cite{Gu2} and Definition 2.1 in \cite{U8}]
Suppose that $n\geq3$. The affine Yangian $Y_{\hbar,\ve}(\widehat{\mathfrak{sl}}(n))$ is isomorphic to the associative algebra generated by 
\begin{equation*}
\{X_{i,r}^{+}, X_{i,r}^{-}, H_{i,r}\mid 0\leq i\leq n-1, r\in\mathbb{Z}_{\geq0}\}
\end{equation*} 
subject to the following defining relations:
\begin{gather}
[H_{i,r}, H_{j,s}] = 0,\label{Eq1.1}\\
[X_{i,r}^{+}, X_{j,s}^{-}] = \delta_{i,j} H_{i,r+s},\label{Eq1.2}\\
[H_{i,0}, X_{j,r}^{\pm}] = \pm a_{i,j} X_{j,r}^{\pm},\label{Eq1.4}\\
[H_{i,r+1}^{\pm}, X_{j, s}^{\pm}] - [H_{i, r}^{\pm}, X_{j, s+1}^{\pm}] = \pm a_{i,j}\dfrac{\hbar}{2} \{H_{i,r}^{\pm}, X_{j,s}^{\pm}\}\text{ if }(i,j)\neq(0,n-1),(n-1,0),\label{Eq1.5}\\
[H_{0, r+1}^{\pm}, X_{n-1, s}^{\pm}] - [H_{0, r}^{\pm}, X_{n-1, s+1}^{\pm}]= \mp\dfrac{\hbar}{2} \{H_{0, r}^{\pm}, X_{n-1, s}^{\pm}\} + (\ve+\dfrac{n}{2}\hbar) [H_{0, r}^{\pm}, X_{n-1, s}^{\pm}],\label{Eq1.6}\\
[H_{n-1, r+1}^{\pm}, X_{0, s}^{\pm}] - [H_{n-1, r}^{\pm}, X_{0, s+1}^{\pm}]= \mp\dfrac{\hbar}{2} \{H_{n-1, r}^{\pm}, X_{0, s}^{\pm}\} - (\ve+\dfrac{n}{2}\hbar) [H_{n-1, r}^{\pm}, X_{0, s}^{\pm}],\label{Eq1.7}\\
[X_{i,r+1}^{\pm}, X_{j, s}^{\pm}] - [X_{i, r}^{\pm}, X_{j, s+1}^{\pm}] = \pm a_{i,j}\dfrac{\hbar}{2} \{X_{i,r}^{\pm}, X_{j,s}^{\pm}\}\text{ if }(i,j)\neq(0,n-1),(n-1,0),\label{Eq1.8}\\
[X_{0, r+1}^{\pm}, X_{n-1, s}^{\pm}] - [X_{0, r}^{\pm}, X_{n-1, s+1}^{\pm}]= \mp\dfrac{\hbar}{2} \{X_{0, r}^{\pm}, X_{n-1, s}^{\pm}\} + (\ve+\dfrac{n}{2}\hbar) [X_{0, r}^{\pm}, X_{n-1, s}^{\pm}],\label{Eq1.9}\\
\sum_{\sigma\in S_{1-a_{i,j}}}(\ad X_{i,r_{\sigma(1)}}^{\pm})\cdots(\ad X_{i,r_{\sigma(1-a_{i,j})}}^{\pm}) (X_{j,s}^{\pm})= 0 \ \text{ if }i \neq j, \label{Eq1.10}
\end{gather}
where $S_l$ is the permutaion group of degree $l$.
\end{Definition}
In \cite{GNW}, Guay-Nakajima-Wendlandt gave the presentation of $Y_{\hbar,\ve}(\widehat{\mathfrak{sl}}(n))$ whose number of generators is finite.
\begin{Proposition}\label{Prop32}
Suppose that $n\geq3$. The affine Yangian $Y_{\hbar,\ve}(\widehat{\mathfrak{sl}}(n))$ is isomorphic to the associative algebra generated by $X_{i,r}^{+}, X_{i,r}^{-}, H_{i,r}$ $(i \in \{0,1,\cdots, n-1\}, r = 0,1)$ subject to the following defining relations:
\begin{gather}
[H_{i,r}, H_{j,s}] = 0,\label{Eq2.1}\\
[X_{i,0}^{+}, X_{j,0}^{-}] = \delta_{i,j} H_{i, 0},\label{Eq2.2}\\
[X_{i,1}^{+}, X_{j,0}^{-}] = \delta_{i,j} H_{i, 1} = [X_{i,0}^{+}, X_{j,1}^{-}],\label{Eq2.3}\\
[H_{i,0}, X_{j,r}^{\pm}] = \pm a_{i,j} X_{j,r}^{\pm},\label{Eq2.4}\\
[\tilde{H}_{i,1}, X_{j,0}^{\pm}] = \pm a_{i,j}X_{j,1}^{\pm},\text{ if }(i,j)\neq(0,n-1),(n-1,0),\label{Eq2.5}\\
[\tilde{H}_{0,1}, X_{n-1,0}^{\pm}] = \mp \left(X_{n-1,1}^{\pm}+(\ve+\dfrac{n}{2}\hbar) X_{n-1, 0}^{\pm}\right),\label{Eq2.6}\\
[\tilde{H}_{n-1,1}, X_{0,0}^{\pm}] = \mp \left(X_{0,1}^{\pm}-(\ve+\dfrac{n}{2}\hbar) X_{0, 0}^{\pm}\right),\label{Eq2.7}\\
[X_{i, 1}^{\pm}, X_{j, 0}^{\pm}] - [X_{i, 0}^{\pm}, X_{j, 1}^{\pm}] = \pm a_{i,j}\dfrac{\hbar}{2} \{X_{i, 0}^{\pm}, X_{j, 0}^{\pm}\}\text{ if }(i,j)\neq(0,n-1),(n-1,0),\label{Eq2.8}\\
[X_{0, 1}^{\pm}, X_{n-1, 0}^{\pm}] - [X_{0, 0}^{\pm}, X_{n-1, 1}^{\pm}]= \mp\dfrac{\hbar}{2} \{X_{0, 0}^{\pm}, X_{n-1, 0}^{\pm}\} + (\ve+\dfrac{n}{2}\hbar) [X_{0, 0}^{\pm}, X_{n-1, 0}^{\pm}],\label{Eq2.9}\\
(\ad X_{i,0}^{\pm})^{1+|a_{i,j}|} (X_{j,0}^{\pm})= 0 \ \text{ if }i \neq j, \label{Eq2.10}
\end{gather}
where $\tilde{H}_{i,1}=H_{i,1}-\dfrac{\hbar}{2}H_{i,0}^2$.
\end{Proposition}

By the definition of the affine Yangian $Y_{\hbar,\ve}(\widehat{\mathfrak{sl}}(n))$, we find that there exists a natural homomorphism from the universal enveloping algebra of $\widehat{\mathfrak{sl}}(n)$ to $Y_{\hbar,\ve}(\widehat{\mathfrak{sl}}(n))$. In order to simplify the notation, we denote the image of $x\in U(\widehat{\mathfrak{sl}}(n))$ via this homomorphism by $x$. We also denote by $E_{i,j}$ a matrix unit whose $(i,j)$ componet is 1 and other components are zero. 

We introduce a grading on $Y_{\hbar,\ve}(\widehat{\mathfrak{sl}}(n))$ by
\begin{equation*}
\text{deg}(H_{i,r})=0,\ \text{deg}(X^\pm_{i,r})=\begin{cases}
\pm1&\text{ if }i=0,\\
0&\text{ if }i\neq0.
\end{cases}
\end{equation*}
We denote the standard degreewise completion of $Y_{\hbar,\ve}(\widehat{\mathfrak{sl}}(n))$ by $\widetilde{Y}_{\hbar,\ve}(\widehat{\mathfrak{sl}}(n))$. 

\begin{Theorem}[Theorem~3.1 in \cite{U8}]\label{Main}
There exists an algebra homomorphism
\begin{equation*}
\Psi^{n,n+1}\colon Y_{\hbar,\ve}(\widehat{\mathfrak{sl}}(n))\to \widetilde{Y}_{\hbar,\ve}(\widehat{\mathfrak{sl}}(n+1))
\end{equation*}
determined by
\begin{gather*}
\Psi^{n,n+1}(H_{i,0})=\begin{cases}
H_{0,0}+H_{n,0}&\text{ if }i=0,\\
H_{i,0}&\text{ if }i\neq 0,
\end{cases}\\
\Psi^{n,n+1}(X^+_{i,0})=\begin{cases}
E_{n,1}t&\text{ if }i=0,\\
E_{i,i+1}&\text{ if }i\neq 0,
\end{cases}\ 
\Psi^{n,n+1}(X^-_{i,0})=\begin{cases}
E_{1,n}t^{-1}&\text{ if }i=0,\\
E_{i+1,i}&\text{ if }i\neq 0,
\end{cases}
\end{gather*}
and
\begin{align*}
\Psi^{n,n+1}(H_{i,1})&= H_{i,1}-\hbar\displaystyle\sum_{s \geq 0} \limits E_{i,n+1}t^{-s-1} E_{n+1,i}t^{s+1}+\hbar\displaystyle\sum_{s \geq 0}\limits E_{i+1,n+1}t^{-s-1} E_{n+1,i+1}t^{s+1},\\
\Psi^{n,n+1}(X^+_{i,1})&=X^+_{i,1}-\hbar\displaystyle\sum_{s \geq 0}\limits E_{i,n+1}t^{-s-1} E_{n+1,i+1}t^{s+1},\\
\Psi^{n,n+1}(X^-_{i,1})&=X^-_{i,1}-\hbar\displaystyle\sum_{s \geq 0}\limits E_{i+1,n+1}t^{-s-1} E_{n+1,i}t^{s+1}
\end{align*}
for $i\neq 0$. 
\end{Theorem}
Let $m$ be an integer greater than $n$. By combining the homomorphisms $\Psi^{n,n+1},\cdots,\Psi^{m-1,m}$, we obtain a homomorphism
\begin{equation*}
\Psi^{n,m}=\Psi^{m-1,m}\circ\cdots\circ\Psi^{n,n+1}\colon Y_{\hbar,\ve}(\widehat{\mathfrak{sl}}(n))\to \widetilde{Y}_{\hbar,\ve}(\widehat{\mathfrak{sl}}(m))
\end{equation*}
determined by
\begin{gather*}
\Psi^{n,m}(H_{i,0})=\begin{cases}
H_{0,0}+\sum_{u=n}^{m-1}\limits H_{u,0}&\text{ if }i=0,\\
H_{i,0}&\text{ if }i\neq 0,
\end{cases}\\
\Psi^{n,m}(X^+_{i,0})=\begin{cases}
E_{n,1}t&\text{ if }i=0,\\
E_{i,i+1}&\text{ if }i\neq 0,
\end{cases}\ 
\Psi^{n,m}(X^-_{i,0})=\begin{cases}
E_{1,n}t^{-1}&\text{ if }i=0,\\
E_{i+1,i}&\text{ if }i\neq 0,
\end{cases}
\end{gather*}
and
\begin{align*}
\Psi^{n,m}(H_{i,1})&= H_{i,1}-\hbar\displaystyle\sum_{s \geq 0} \limits\sum_{u=n+1}^m\limits E_{i,u}t^{-s-1} E_{u,i}t^{s+1}+\hbar\displaystyle\sum_{s \geq 0}\limits\sum_{u=n+1}^m\limits E_{i+1,u}t^{-s-1} E_{u,i+1}t^{s+1},\\
\Psi^{n,m}(X^+_{i,1})&=X^+_{i,1}-\hbar\displaystyle\sum_{s \geq 0}\limits\sum_{u=n+1}^m\limits E_{i,u}t^{-s-1} E_{u,i+1}t^{s+1},\\
\Psi^{n,m}(X^-_{i,1})&=X^-_{i,1}-\hbar\displaystyle\sum_{s \geq 0}\limits\sum_{u=n+1}^m\limits E_{i+1,u}t^{-s-1} E_{u,i}t^{s+1}
\end{align*}
for $i\neq 0$. 
\section{$W$-algebras of type $A$}
We fix some notations for vertex algebras. For a vertex algebra $V$, we denote the generating field associated with $x\in V$ by $x(z)=\displaystyle\sum_{n\in\mathbb{Z}}\limits x_{(n)}z^{-n-1}$. We also denote the OPE of $V$ by
\begin{equation*}
x(z)y(w)\sim\displaystyle\sum_{s\geq0}\limits \dfrac{(x_{(s)}y)(w)}{(z-w)^{s+1}}
\end{equation*}
for all $x,y\in V$. We denote the vacuum vector (resp.\ the translation operator) by $|0\rangle$ (resp.\ $\partial$).

We set
\begin{equation*}
N=\displaystyle\sum_{i=1}^s\limits l_iq_i,\qquad q_1>q_2>\cdots>q_s>q_{s+1}=0,\ l_i\in\mathbb{Z}_{>0}.
\end{equation*}
We set a basis of $\mathfrak{gl}(N)$ as $\mathfrak{gl}(N)=\displaystyle\bigoplus_{1\leq i,j\leq N}\limits\mathbb{C}e_{i,j}$. 
For $k\in\mathbb{C}$, we take a symmetric bilinear form on $\mathfrak{gl}(N)$ determined by
\begin{equation*}
(e_{i,j}|e_{p,q})=k\delta_{i,q}\delta_{p,j}+\delta_{i,j}\delta_{p,q}.
\end{equation*}
We set two mappings
\begin{equation*}
\col:\{1,2,\cdots,N\}\to\{1,2,\cdots,s\},\ \row:\{1,2,\cdots,N\}\to\{1,2,\cdots,q_1\}
\end{equation*}
by
\begin{gather*}
i=\sum_{u=1}^v\limits l_uq_u+(\col(i)-\sum_{u=1}^vl_u-1)q_{v+1}+(\row(i)-q_1+q_{v+1}),\ 1\leq\row(i)\leq q_{v+1}
\end{gather*}
in the case that $\sum_{u=1}^v\limits l_uq_u<i\leq\sum_{u=1}^{v+1}\limits l_uq_u$.
We also set two mappings
\begin{gather*}
\hat{\ }\colon\{1,2,\cdots,N\}\setminus\{\sum_{u=1}^v\limits l_uq_u-w+1\mid 1\leq v\leq l,1\leq w\leq q_{v+1}\}\to\{1,2,\cdots,N\}\ i\mapsto\hat{i},\\
\tilde{\ }\colon\{q_1+1,q_1+2,\cdots,N\}\to\{1,2,\cdots,N\}\ j\mapsto\tilde{j}
\end{gather*}
by
\begin{gather*}
\col(\hat{i})=\col(i)+1,\row(\hat{i})=\row(i),
\col(\tilde{j})=\col(j)-1,\row(\tilde{j})=\row(j).
\end{gather*}
For all $1\leq i,j\leq N$, we take $1\leq \hat{i},\tilde{j}\leq N$ as

For example, in the case that $q_1=2,q_2=1,l_1=2,l_2=1$, we obtain
\begin{gather*}
\col(3)=2,\ \row(3)=1,\ \widehat{3}\text{ does not exist},\ \widetilde{3}=1.
\end{gather*}

We take a nilpotent element $f_{q_1,q_2,\cdots,q_s}^{l_1,l_2\cdots,l_s}\in\mathfrak{gl}(N)$ as 
\begin{equation*}
f_{q_1,q_2,\cdots,q_s}^{l_1,l_2\cdots,l_s}=\sum_{1\leq j\leq N}\limits e_{\hat{j},j}.
\end{equation*}
We note that $f_{q_1,q_2,\cdots,q_s}^{l_1,l_2\cdots,l_s}$ is of type $(l_1^{q_1-q_2},(l_1+l_2)^{q_2-q_3},\cdots,(\sum_{v=1}^{s-1}l_v)^{q_{s-1}-q_{s}},(\sum_{v=1}^{s}l_v)^{q_s})$.
We take an $\mathfrak{sl}(2)$-triple $(x,e,f_{q_1,q_2,\cdots,q_u}^{l_1,l_2\cdots,l_u})$, that is,
\begin{equation*}
x=[e,f_{q_1,q_2,\cdots,q_s}^{l_1,l_2\cdots,l_s}],\ [x,e]=e,\ [x,f_{q_1,q_2,\cdots,q_s}^{l_1,l_2\cdots,l_s}]=-f_{q_1,q_2,\cdots,q_s}^{l_1,l_2\cdots,l_s}
\end{equation*}
satisfying that
\begin{equation*}
\{y\in\mathfrak{gl}(N)|[x,y]=ay\}=\bigoplus_{\col(j)-\col(i)=a}\limits\mathbb{C} e_{i,j}.
\end{equation*}
We can set the grading of $\mathfrak{gl}(N)$ by $\text{deg}(e_{i,j})=\col(j)-\col(i)$ (see Section 7 in \cite{BK}). We will denote $\bigoplus_{\col(j)-\col(i)=u}\limits \mathbb{C}e_{i,j}$ by $\mathfrak{g}_u$.

We consider two vertex algebras. The first one is the universal affine vertex algebra associated with a Lie subalgebra
\begin{align*}
\mathfrak{b}&=\bigoplus_{u\leq0}\limits\mathfrak{g}_u=\bigoplus_{\substack{1\leq i,j\leq N\\\col(i)\geq\col(j)}}\limits \mathbb{C}e_{i,j}\subset\mathfrak{gl}(N)
\end{align*}
and its inner product
\begin{equation*}
\kappa(e_{i,j},e_{p,q})=\zeta_{\col(i)}\delta_{i,q}\delta_{p,j}+\delta_{i,j}\delta_{p,q},
\end{equation*}
where $\zeta_g=k+N-q_u$ if $\sum_{v=1}^{u-1}l_v<g\leq\sum_{v=1}^{u}l_v$.

The second one is the universal affine vertex algebra associated with a Lie superalgebra $\mathfrak{a}=\mathfrak{b}\oplus\displaystyle\bigoplus_{\substack{1\leq i,j\leq N\\\col(i)>\col(j)}}\limits\mathbb{C}\psi_{i,j}$  with the following commutator relations;
\begin{align*}
[e_{i,j},\psi_{p,q}]&=\delta_{j,p}\psi_{i,q}-\delta_{i,q}\psi_{p,j},\\
[\psi_{i,j},\psi_{p,q}]&=0,
\end{align*}
where $e_{i,j}$ is an even element and $\psi_{i,j}$ is an odd element.
We set the inner product on $\mathfrak{a}$ such that
\begin{gather*}
\widetilde{\kappa}(e_{i,j},e_{p,q})=\kappa(e_{i,j},e_{p,q}),\qquad\widetilde{\kappa}(e_{i,j},\psi_{p,q})=\widetilde{\kappa}(\psi_{i,j},\psi_{p,q})=0.
\end{gather*}
By the definition of $V^{\widetilde{\kappa}}(\mathfrak{a})$ and $V^\kappa(\mathfrak{b})$, $V^{\widetilde{\kappa}}(\mathfrak{a})$ contains $V^\kappa(\mathfrak{b})$.

By the Poincare-Birkhoff-Witt theorem, we can identify $V^{\widetilde{\kappa}}(\mathfrak{a})$ (resp. $V^\kappa(\mathfrak{b})$) with $U(\mathfrak{a}t^{-1}[t^{-1}])$ (resp. $U(\mathfrak{b}t^{-1}[t^{-1}])$). In order to simplify the notation, we denote the generating field $(xt^{-1})(z)$ as $x(z)$. By the definition of $V^{\widetilde{\kappa}}(\mathfrak{a})$, generating fields $u(z)$ and $v(z)$ satisfy the OPE
\begin{gather}
x(z)y(w)\sim\dfrac{[x,y](w)}{z-w}+\dfrac{\widetilde{\kappa}(x,y)}{(z-w)^2}\label{OPE1}
\end{gather}
for all $x,y\in\mathfrak{a}$. 

For all $x\in \mathfrak{a}$, let $x[-\omega]$ be $xt^{-\omega}\in U(\mathfrak{a}t^{-1}[t^{-1}])$. In this section, we regard $V^{\widetilde{\kappa}}(\mathfrak{a})$ (resp.\ $V^\kappa(\mathfrak{b})$) as a non-associative superalgebra whose product $\cdot$ is defined by
\begin{equation*}
x[-\omega]\cdot y[-\theta]=(x[-\omega])_{(-1)}y[-\theta].
\end{equation*}
We sometimes omit $\cdot$ in order to simplify the notation. By Kac-Wakimoto \cite{KW1} and \cite{KW2}, a $W$-algebra $\mathcal{W}^k(\mathfrak{gl}(N),f_{q_1,q_2,\cdots,q_s}^{l_1,l_2\cdots,l_s})$ can be realized as a subalgebra of $V^\kappa(\mathfrak{b})$.

Let us define an odd differential $d_0 \colon V^{\kappa}(\mathfrak{b})\to V^{\widetilde{\kappa}}(\mathfrak{a})$ determined by
\begin{gather}
d_01=0,\\
[d_0,\partial]=0,\label{ee5800}
\end{gather}
\begin{align}
[d_0,e_{i,j}[-1]]
&=\sum_{\substack{\col(i)>\col(r)\geq\col(j)}}\limits e_{r,j}[-1]\psi_{i,r}[-1]-\sum_{\substack{\col(j)<\col(r)\leq\col(i)}}\limits \psi_{r,j}[-1]e_{i,r}[-1]\nonumber\\
&\quad+\delta(\col(i)>\col(j))\zeta_{\col(i)}\psi_{i,j}[-2]+\psi_{\hat{i},j}[-1]-\psi_{i,\tilde{j}}[-1].\label{ee1}
\end{align}
By using Theorem 2.4 in \cite{KRW}, we can define the $W$-algebra $\mathcal{W}^k(\mathfrak{gl}(N),f_{q_1,q_2,\cdots,q_s}^{l_1,l_2\cdots,l_s})$ as follows.
\begin{Definition}\label{T125}
The $W$-algebra $\mathcal{W}^k(\mathfrak{gl}(N),f_{q_1,q_2,\cdots,q_s}^{l_1,l_2\cdots,l_s})$ is the vertex subalgebra of $V^\kappa(\mathfrak{b})$ defined by
\begin{equation*}
\mathcal{W}^k(\mathfrak{gl}(N),f_{q_1,q_2,\cdots,q_s}^{l_1,l_2\cdots,l_s})=\{y\in V^\kappa(\mathfrak{b})\subset V^{\widetilde{\kappa}}(\mathfrak{a})\mid d_0(y)=0\}.
\end{equation*}
\end{Definition}
We will give some kinds of elements of $\mathcal{W}^k(\mathfrak{gl}(N),f_{q_1,q_2,\cdots,q_s}^{l_1,l_2\cdots,l_s})$.
\begin{Theorem}\label{Generators}
We suppose that $p,q$ are integers satisfying that $q_1-q_{a}\leq q<q_1-q_{a+1}, q_1-q_{b}\leq p<q_1-q_{b+1}$ for some integers $1\leq b\leq a\leq s$. 
Let us set two kinds of elements of $V^\kappa(\mathfrak{b})$:
\begin{align*}
W^{(1)}_{p,q}&=\sum_{\substack{1\leq i,j\leq N,\\\row(i)=p,\row(j)=q,\\\col(i)=\col(j)}}e_{i,j}[-1],\\
W^{(2)}_{p,q}&=\sum_{\substack{\col(i)=\col(j)+1\\\row(i)=p,\row(j)=q}}e_{i,j}[-1]-\sum_{\substack{\col(i)=\col(j)\\\row(i)=p,\row(j)=q}}\gamma_{\col(i)}e_{i,j}[-2]\\
&\quad+\sum_{\substack{\col(u)=\col(j)<\col(i)=\col(v)\\\row(u)=\row(v)>q_1-q_a\\\row(i)=p,\row(j)=q}}\limits e_{u,j}[-1]e_{i,v}[-1]-\sum_{\substack{\col(u)=\col(j)\geq\col(i)=\col(v)\\\row(u)=\row(v)\leq q_1-q_a\\\row(i)=p,\row(j)=q}}\limits e_{u,j}[-1]e_{i,v}[-1],
\end{align*}
where $\gamma_c=\sum_{b=c+1}^{l_1+\cdots+l_s}\limits \zeta_{b}$.
Then, the $W$-algebra $\mathcal{W}^k(\mathfrak{gl}(N),f_{q_1,q_2,\cdots,q_s}^{l_1,l_2\cdots,l_s})$ contains $W^{(1)}_{p,q}$ and $W^{(2)}_{p,q}$.
\end{Theorem}
\begin{proof}
By \eqref{ee1}, if $\col(i)=\col(j)$, we obtain
\begin{align}
d_0(e_{i,j}[-1])
&=\psi_{\widehat{i},j}[-1]-\psi_{i,\widetilde{j}}[-1].\label{ee307}
\end{align}
Similarly, by \eqref{ee1}, if $\col(i)=\col(j)+1$, we also have
\begin{align}
&\quad d_0(e_{i,j}[-1])\nonumber\\
&=\sum_{\substack{\col(r)=\col(j)}}\limits e_{r,j}[-1]\psi_{i,r}[-1]-\sum_{\substack{\col(r)=\col(i)}}\limits \psi_{r,j}[-1]e_{i,r}[-1]\nonumber\\
&\quad+\zeta_{\col(i)}\psi_{i,j}[-2]+\psi_{\hat{i},j}[-1]-\psi_{i,\tilde{j}}[-1].\label{ee2}
\end{align}

It is enough to show that $d_0(W^{(r)}_{p,q})=0$. First, we show the case that $r=1$.
By \eqref{ee307}, we obtain
\begin{align}
d_0(W^{(1)}_{p,q})&=\sum_{\substack{1\leq i,j\leq N,\\\row(i)=p,\row(j)=q,\\\col(i)=\col(j)}}(\psi_{\widehat{i},j}[-1]-\psi_{i,\widetilde{j}}[-1])\nonumber\\
&=\sum_{\substack{1\leq i,j\leq N,\\\row(i)=p,\row(j)=q,\\\col(i)=\col(j)}}\psi_{\widehat{i},j}[-1]-\sum_{\substack{1\leq i,j\leq N,\\\row(i)=p,\row(j)=q,\\\col(i)=\col(j)}}\psi_{i,\widetilde{j}}[-1]\nonumber\\
&=0,\label{ee380}
\end{align}
where the last equality is due to the assumption that $1\leq b\leq a\leq l$.

Next, we show the case when $r=2$. 
By the definition of $W^{(2)}_{p,q}$, we can rewrite $d_0(W^{(2)}_{p,q})$ as
\begin{align}
&\sum_{\substack{\col(i)=\col(j)+1\\\row(i)=p,\row(j)=q}}d_0(e_{i,j}[-1])-\sum_{\substack{\col(i)=\col(j)\\\row(i)=p,\row(j)=q}}\gamma_{\col(i)}d_0(e_{i,j}[-2])\nonumber\\
&\quad+\sum_{\substack{\col(u)=\col(j)<\col(i)=\col(v)\\\row(u)=\row(v)>q_1- q_a\\\row(i)=p,\row(j)=q}}\limits d_0(e_{u,j}[-1])e_{i,v}[-1]\nonumber\\
&\quad+\sum_{\substack{\col(u)=\col(j)<\col(i)=\col(v)\\\row(u)=\row(v)>q_1- q_a\\\row(i)=p,\row(j)=q}}\limits e_{u,j}[-1]d_0(e_{i,v}[-1])\nonumber\\
&\quad-\sum_{\substack{\col(u)=\col(j)\geq\col(i)=\col(v)\\\row(u)=\row(v)\leq q_1-q_a\\\row(i)=p,\row(j)=q}}\limits d_0(e_{u,j}[-1])e_{i,v}[-1]\nonumber\\
&\quad-\sum_{\substack{\col(u)=\col(j)\geq\col(i)=\col(v)\\\row(u)=\row(v)\leq q_1-q_a\\\row(i)=p,\row(j)=q}}\limits e_{u,j}[-1]d_0(e_{i,v}[-1]).\label{ee3}
\end{align}
By \eqref{ee2}, we obtain
\begin{align}
&\quad\text{the first term of \eqref{ee3}}\nonumber\\
&=\sum_{\substack{\col(i)=\col(j)+1\\\row(i)=p,\row(j)=q}}\sum_{\substack{\col(r)=\col(j)}}\limits e_{r,j}[-1]\psi_{i,r}[-1]-\sum_{\substack{\col(i)=\col(j)+1\\\row(i)=p,\row(j)=q}}\sum_{\substack{\col(r)=\col(i)}}\limits \psi_{r,j}[-1]e_{i,r}[-1]\nonumber\\
&\quad+\sum_{\substack{\col(i)=\col(j)+1\\\row(i)=p,\row(j)=q}}\zeta_{\col(i)}\psi_{i,j}[-2]+\sum_{\substack{\col(i)=\col(j)+1\\\row(i)=p,\row(j)=q}}(\psi_{\hat{i},j}[-1]-\psi_{i,\tilde{j}}[-1]).\label{ee5}
\end{align}
Similarly to the proof of $d_0(W^{(1)}_{i,j})=0$, we find that the last term of the right hand side of \eqref{ee5} is equal to zero. Then, we have
\begin{align}
&\quad\text{the first term of \eqref{ee3}}\nonumber\\
&=\sum_{\substack{\col(i)=\col(j)+1\\\row(i)=p,\row(j)=q}}\sum_{\substack{\col(r)=\col(j)}}\limits e_{r,j}[-1]\psi_{i,r}[-1]-\sum_{\substack{\col(i)=\col(j)+1\\\row(i)=p,\row(j)=q}}\sum_{\substack{\col(r)=\col(i)}}\limits \psi_{r,j}[-1]e_{i,r}[-1]\nonumber\\
&\quad+\sum_{\substack{\col(i)=\col(j)+1\\\row(i)=p,\row(j)=q}}\zeta_{\col(i)}\psi_{i,j}[-2].\label{ee5.1}
\end{align}

By \eqref{ee307} and \eqref{ee5800}, we obtain
\begin{align}
&\quad\text{the second term of \eqref{ee3}}\nonumber\\
&=-\sum_{\substack{\col(i)=\col(j)\\\row(i)=p,\row(j)=q}}\gamma_{\col(i)}(\psi_{\widehat{i},j}[-2]-\psi_{i,\widetilde{j}}[-2])\nonumber\\
&=\sum_{\substack{\col(i)=\col(j)\\\row(i)=p,\row(j)=q}}(\gamma_{\col(\hat{i})}-\gamma_{\col(i)})\psi_{\widehat{i},j}[-2]\nonumber\\
&=-\sum_{\substack{\col(i)=\col(j)\\\row(i)=p,\row(j)=q}}\limits\zeta_{\col(\hat{i})}\psi_{\hat{i},j}[-2].\label{ee6}
\end{align}
By \eqref{ee307}, we obtain
\begin{align}
&\quad\text{the third term of the right hand side of \eqref{ee3}}\nonumber\\
&=\sum_{\substack{\col(u)=\col(j)<\col(i)=\col(v)\\\row(u)=\row(v)>q_1- q_a\\\row(i)=p,\row(j)=q}}\limits (\psi_{\hat{u},j}[-1]-\psi_{u,\tilde{j}}[-1])e_{i,v}[-1]\nonumber\\
&=\sum_{\substack{\col(u)+1=\col(j)+1=\col(i)=\col(v)\\\row(u)=\row(v)>q_1- q_a\\\row(i)=p,\row(j)=q}}\limits \psi_{\hat{u},j}[-1]e_{i,v}[-1]\label{ee7},\\
&\quad\text{the 4-th term of the right hand side of \eqref{ee3}}\nonumber\\
&=\sum_{\substack{\col(u)=\col(j)<\col(i)=\col(v)\\\row(u)=\row(v)>q_1- q_a\\\row(i)=p,\row(j)=q}}\limits e_{u,j}[-1](\psi_{\hat{i},v}[-1]-\psi_{i,\tilde{v}}[-1])\nonumber\\
&=-\sum_{\substack{\col(u)+1=\col(j)+1=\col(i)=\col(v)\\\row(u)=\row(v)>q_1- q_a\\\row(i)=p,\row(j)=q}}\limits e_{u,j}[-1]\psi_{i,\tilde{v}}[-1],\label{ee8}\\
&\quad\text{the 5-th term of the right hand side of \eqref{ee3}}\nonumber\\
&=-\sum_{\substack{\col(u)=\col(j)\geq\col(i)=\col(v)\\q_1-q_{\col(j)}<\row(u)=\row(v)\leq q_1-q_a\\\row(i)=p,\row(j)=q}}\limits (\psi_{\hat{u},j}[-1]-\psi_{u,\tilde{j}}[-1])e_{i,v}[-1]\nonumber\\
&=\sum_{\substack{\col(u)=\col(j)=\col(i)+1=\col(v)+1\\q_1-q_{\col(j)}<\row(u)=\row(v)\leq q_1-q_a\\\row(i)=p,\row(j)=q}}\limits\psi_{u,\tilde{j}}[-1]e_{i,v}[-1],\label{ee9}\\
&\quad\text{the 6-th term of the right hand side of \eqref{ee3}}\nonumber\\
&=-\sum_{\substack{\col(u)=\col(j)\geq\col(i)=\col(v)\\q_1-q_{\col(j)}<\row(u)=\row(v)\leq q_1-q_a\\\row(i)=p,\row(j)=q}}\limits e_{u,j}[-1](\psi_{\hat{i},v}[-1]-\psi_{i,\tilde{v}}[-1])\nonumber\\
&=-\sum_{\substack{\col(u)=\col(j)=\col(i)+1=\col(v)+1\\q_1-q_{\col(j)}<\row(u)=\row(v)\leq q_1-q_a\\\row(i)=p,\row(j)=q}}\limits e_{u,j}[-1]\psi_{\hat{i},v}[-1].\label{ee10}
\end{align}
Here after, in order to simplify the notation, let us denote the $i$-th term of (the number of the equation) by 
$\text{(the number of the equation)}_i$. By a direct computation, we obtain
\begin{gather*}
\eqref{ee5}_1+\eqref{ee8}+\eqref{ee10}=0,\\
\eqref{ee5}_2+\eqref{ee7}+\eqref{ee9}=0,\\
\eqref{ee5}_3+\eqref{ee6}=0.
\end{gather*}
Then, adding \eqref{ee5}-\eqref{ee10}, we obtain $d_0(W^{(2)}_{p,q})=0$.
\end{proof}
\begin{Remark}\label{ren}
The element $W^{(r)}_{i,j}$ for $1\leq i,j\leq q_s$ in \cite{U7} coincides with $W^{(r)}_{q_1-i,q_1-j}$ in Theorem~\ref{Generators}.
\end{Remark}
We write down some OPEs of $W^{(1)}_{i,j}$ and $W^{(1)}_{p,q}$.
\begin{Theorem}\label{Tho1}
We suppose that $p,q,i,j$ are integers satisfying that $q_1-q_{a_1}\leq q<q_1-q_{a_1+1},\ q_1-q_{b_1}\leq p<q_1-q_{b_1+1},\ q_1-q_{a_2}\leq j<q_1-q_{a_2+1},\ q_1-q_{b_2}\leq i<q_1-q_{b_2+1}$ for some integers $1\leq b_1\leq a_1\leq s$ and $1\leq b_2\leq a_2\leq s$. 

\textup{(1)}\ 
The following equation holds:
\begin{gather*}
(W^{(1)}_{i,j})_{(0)}W^{(1)}_{p,q}=\delta_{j,p}W^{(1)}_{i,q}-\delta_{i,q}W^{(1)}_{p,j}.
\end{gather*}
\textup{(2)}\ The following equation holds:
\begin{align*}
(W^{(1)}_{i,j})_{(0)}W^{(2)}_{p,q}
&=\delta_{j,p}W^{(2)}_{i,q}-\delta_{i,q}W^{(2)}_{p,j}+\delta(j>q_1-q_a,i\leq q_1-q_a)(W^{(1)}_{i,q})_{(-1)}W^{(1)}_{p,j}.
\end{align*}
\end{Theorem}
The proof can be given by a direct computation.
\section{Affine Yangians and non-rectangular $W$-algebras}
Let us recall the definition of a universal enveloping algebra of a vertex algebra in the sense of \cite{FZ} and \cite{MNT}.
For any vertex algebra $V$, let $L(V)$ be the Borcherds Lie algebra, that is,
\begin{align}
 L(V)=V{\otimes}\mathbb{C}[t,t^{-1}]/\text{Im}(\partial\otimes\id +\id\otimes\frac{d}{d t})\label{844},
\end{align}
where the commutation relation is given by
\begin{align*}
 [ut^a,vt^b]=\sum_{r\geq 0}\begin{pmatrix} a\\r\end{pmatrix}(u_{(r)}v)t^{a+b-r}
\end{align*}
for all $u,v\in V$ and $a,b\in \mathbb{Z}$. Now, we define the universal enveloping algebra of $V$.
\begin{Definition}[Section~6 in \cite{MNT}]\label{Defi}
We define $\mathcal{U}(V)$ as the quotient algebra of the standard degreewise completion of the universal enveloping algebra of $L(V)$ by the completion of the two-sided ideal generated by
\begin{gather}
(u_{(a)}v)t^b-\sum_{i\geq 0}
\begin{pmatrix}
 a\\i
\end{pmatrix}
(-1)^i(ut^{a-i}vt^{b+i}-(-1)^avt^{a+b-i}ut^{i}),\label{241}\\
|0\rangle t^{-1}-1.
\end{gather}
We call $\mathcal{U}(V)$ the universal enveloping algebra of $V$.
\end{Definition}
The projection map from $\mathfrak{b}$ to $\bigotimes_{i=1}^l\limits \mathfrak{gl}(q_i)$ induces the injective homomorphism called the Miura map (see \cite{KW1}):
\begin{align*}
\mu\colon \mathcal{W}^k(\mathfrak{gl}(N),f)\to V^{\kappa}(\bigotimes_{i=1}^{s}\mathfrak{gl}(q_i)).
\end{align*}
Since $\mathcal{U}(V^\Omega(\mathfrak{gl}(n)))$ is the standard degreewise completion of the affinization of $\mathfrak{gl}(n)$ associated with the invariant bilinear form $\Omega$ on $\mathfrak{gl}(n)$, the Miura map induces an injective homomorphism
\begin{equation*}
\widetilde{\mu}\colon \mathcal{W}^k(\mathfrak{gl}(N),f)\to\widehat{\bigotimes}_{i=1}^{s}U(\widehat{\mathfrak{gl}}(q_i))^{\otimes l_i},
\end{equation*}
where $\widehat{\bigotimes}_{i=1}^{s}U(\widehat{\mathfrak{gl}}(q_i))^{\otimes l_i}$ is the standard degreewise completion of $\bigotimes_{i=1}^{s}U(\widehat{\mathfrak{gl}}(q_i))^{\otimes l_i}$

Here after, in order to simplify the notation, in the case that $q_1-q_a<i,j\leq q_1-q_{a+1}$, we sometimes denote $W^{(r)}_{i,j}$ by $\widetilde{W}^{(r),a}_{i-q_1+q_a,j-q_1+q_a}$.
In Theorem~6.1 of \cite{U7}, we constructed a homomorphism from the affine Yangian associated with $\widehat{\mathfrak{sl}}(q_s)$ to the universal enveloping algebra of 
$\mathcal{W}^{k}(\mathfrak{gl}(N),f_{q_1,q_2,\cdots,q_s}^{l_1,l_2\cdots,l_s})$. Let us set $\widehat{W}^{(1),s}_{i,j}=\widetilde{W}^{(1),s}_{q_s-i,q_s-j}$.
\begin{Theorem}[Theorem 6.1 in \cite{U7}]\label{Ue}
Suppose that $q_s\geq 3$ and a complex number $\ve_s$ satisfies $\dfrac{\ve_s+n\hbar}{\hbar}=k+N$.
Then, there exists an algebra homomorphism
\begin{equation*}
\Phi\colon Y_{\hbar,\ve_s}(\widehat{\mathfrak{sl}}(q_s))\to \mathcal{U}(\mathcal{W}^{k}(\mathfrak{gl}(N),f_{q_1,q_2,\cdots,q_s}^{l_1,l_2\cdots,l_s}))
\end{equation*}
determined by
\begin{gather*}
\Phi(X^+_{i,0})=\begin{cases}
\widehat{W}^{(1),s}_{q_s,1}t&\text{ if }i=0,\\
\widehat{W}^{(1),s}_{i,i+1}&\text{ if }i\neq0,
\end{cases}
\quad \Phi(X^-_{i,0})=\begin{cases}
\widehat{W}^{(1),s}_{q_s,1}t^{-1}&\text{ if }i=0,\\
\widehat{W}^{(1),s}_{i+1,i}&\text{ if }i\neq0,
\end{cases}
\end{gather*}
\begin{align*}
\Phi(H_{i,1})&=
-\hbar(\widehat{W}^{(2),s}_{i,i}t-\widehat{W}^{(2),s}_{i+1,i+1}t)-\dfrac{i}{2}\hbar\Phi(H_{i,0})+\hbar \widehat{W}^{(1),s}_{i,i}\widehat{W}^{(1),s}_{i+1,i+1}\\
&\quad+\hbar\displaystyle\sum_{y \geq 0}  \limits\displaystyle\sum_{u=1}^{i}\limits \widehat{W}^{(1),s}_{i,u}t^{-y}\widehat{W}^{(1),s}_{u,i}t^y\\
&\quad+\hbar\displaystyle\sum_{y \geq 0} \limits\displaystyle\sum_{u=i+1}^{q_s}\limits \widehat{W}^{(1),s}_{i,u}t^{-y-1} \widehat{W}^{(1),s}_{u,i}t^{y+1}\\
&\quad-\hbar\displaystyle\sum_{y \geq 0}\limits\displaystyle\sum_{u=1}^{i}\limits \widehat{W}^{(1),s}_{i+1,u}t^{-y} W^{(1)}_{u,i+1}t^y\\
&\quad-\hbar\displaystyle\sum_{y \geq 0}\limits\displaystyle\sum_{u=i+1}^{q_s} \limits \widehat{W}^{(1),s}_{i+1,u}t^{-y-1} \widehat{W}^{(1),s}_{u,i+1}t^{y+1},\\
\Phi^{q_1,\cdots,q_s}_s(X^+_{i,1})&=
-\hbar \widehat{W}^{(2),s}_{i,i+1}t-\dfrac{i}{2}\hbar\Phi(X_{i,0}^{+})+\hbar\displaystyle\sum_{y \geq 0}\limits\displaystyle\sum_{u=1}^i\limits W^{(1),s}_{i,u}t^{-y} \widehat{W}^{(1),s}_{u,i}t^y\\
&\quad+\hbar\displaystyle\sum_{y \geq 0}\limits\displaystyle\sum_{u=i+1}^{q_s}\limits \widehat{W}^{(1),s}_{i,u}t^{-y-1} \widehat{W}^{(1),s}_{u,i}t^{y+1},\\
\Phi^{q_1,\cdots,q_s}_s(X^-_{i,1})&=
-\hbar \widehat{W}^{(2),s}_{i+1,i}t-\dfrac{i}{2}\hbar\Phi(X_{i,0}^{-})+\hbar\displaystyle\sum_{y\geq 0}\limits\displaystyle\sum_{u=1}^i\limits\widehat{W}^{(1),s}_{i+1,u}t^{-y} \widehat{W}^{(1),s}_{u,i}t^y\\
&\quad+\hbar\displaystyle\sum_{y\geq 0}\limits\displaystyle\sum_{u=i+1}^{q_s}\limits \widehat{W}^{(1),s}_{i+1,u}t^{-y-1} \widehat{W}^{(1),s}_{u,i}t^{y+1}.
\end{align*}
\end{Theorem}
By using $\Psi^{n,m}$, we construct other homomorphisms from the affine Yangian to the universal enveloping algebra of a non-rectangular $W$-algebra.
\begin{Theorem}\label{May}
Let us take an integer $1\leq z\leq s$ and suppose that $q_z-q_{z+1}\geq 3$ and assume that a complex number $\ve_z$ satisfies $\dfrac{\ve_z+q_z\hbar}{\hbar}=k+N$. For $1\leq i,j\leq q_z-q_{z+1}$, let $\widetilde{W}^{(r)}_{i,j}$ be $W^{(r)}_{q_{z}+1-i,q_{z}+1-j}$.Then, there exists an algebra homomorphism
\begin{equation*}
\Phi_z\colon Y_{\hbar,\ve}(\widehat{\mathfrak{sl}}(q_z-q_{z+1}))\to \mathcal{U}(\mathcal{W}^k(\mathfrak{gl}(N),f_{q_1,q_2,\cdots,q_s}^{l_1,l_2\cdots,l_s}))
\end{equation*}
determined by
\begin{gather*}
\Phi_z(X^+_{i,0})=\begin{cases}
\widetilde{W}^{(1),z}_{q_z-q_{z+1},1}t&\text{ if }i=0,\\
\widetilde{W}^{(1),z}_{i,i+1}&\text{ if }i\neq0,
\end{cases}
\quad \Phi_z(X^-_{i,0})=\begin{cases}
\widetilde{W}^{(1),z}_{1,q_z-q_{z+1}}t^{-1}&\text{ if }i=0,\\
\widetilde{W}^{(1),z}_{i+1,i}&\text{ if }i\neq0,
\end{cases}
\end{gather*}
\begin{align*}
\Phi_z(H_{i,1})&=-\hbar(\widetilde{W}^{(2),z}_{i,i}t-\widetilde{W}^{(2),z}_{i+1,i+1}t)-\dfrac{i}{2}\hbar\Phi(H_{i,0})+\hbar \widetilde{W}^{(1),z}_{i,i}\widetilde{W}^{(1),z}_{i+1,i+1}\\
&\quad+\hbar\displaystyle\sum_{y \geq 0}  \limits\displaystyle\sum_{u=1}^{i}\limits \widetilde{W}^{(1),z}_{i,u}t^{-y}\widetilde{W}^{(1),z}_{u,i}t^y+\hbar\displaystyle\sum_{y \geq 0} \limits\displaystyle\sum_{u=i+1}^{q_z-q_{z+1}}\limits \widetilde{W}^{(1),z}_{i,u}t^{-y-1} \widetilde{W}^{(1),z}_{u,i}t^{y+1}\\
&\quad-\hbar\displaystyle\sum_{y \geq 0}\limits\displaystyle\sum_{u=1}^{i}\limits \widetilde{W}^{(1),z}_{i+1,u}t^{-y} \widetilde{W}^{(1),z}_{u,i+1}t^y-\hbar\displaystyle\sum_{y \geq 0}\limits\displaystyle\sum_{u=i+1}^{q_z-q_{z+1}} \limits \widetilde{W}^{(1),z}_{i+1,u}t^{-y-1} \widetilde{W}^{(1),z}_{u,i+1}t^{y+1},\\
\Phi_z(X^+_{i,1})&=-\hbar \widetilde{W}^{(2),z}_{i,i+1}t-\dfrac{i}{2}\hbar\Phi(X_{i,0}^{+})\\
&\quad+\hbar\displaystyle\sum_{y \geq 0}\limits\displaystyle\sum_{u=1}^i\limits \widetilde{W}^{(1),z}_{i,u}t^{-y} \widetilde{W}^{(1),z}_{u,i+1}t^y+\hbar\displaystyle\sum_{y \geq 0}\limits\displaystyle\sum_{u=i+1}^{q_z-q_{z+1}}\limits \widetilde{W}^{(1),z}_{i,u}t^{-y-1} \widetilde{W}^{(1),z}_{u,i+1}t^{y+1},\\
\Phi_z(X^-_{i,1})&=-\hbar \widetilde{W}^{(2),z}_{i+1,i+1+1}t-\dfrac{i}{2}\hbar\Phi(X_{i,0}^{-})\\
&\quad+\hbar\displaystyle\sum_{y \geq 0}\limits\displaystyle\sum_{u=1}^i\limits \widetilde{W}^{(1),z}_{i+1,u}t^{-y} \widetilde{W}^{(1),z}_{u,i}t^y+\hbar\displaystyle\sum_{y \geq 0}\limits\displaystyle\sum_{u=i+1}^{q_z-q_{z+1}}\limits \widetilde{W}^{(1),z}_{i+1,u}t^{-y-1} \widetilde{W}^{(1),z}_{u,i}t^{y+1}
\end{align*}
for $i\neq0$.
\end{Theorem}
\begin{proof}
First, we show the case that $z=s$. Let us set
\begin{equation*}
f_{A,B}^{x,y}(i,j,p,q)=\widetilde{\mu}([\widetilde{W}^{(A),s}_{i,j}t^x,\widetilde{W}^{(B),s}_{p,q}t^y])\text{ for }1\leq i,j,p,q\leq n,A,B=0,1,x,y\in\mathbb{Z}.
\end{equation*}
By the definition of $W^{(r)}_{i,j}$, $f_{A,B}^{x,y}(i,j,p,q)$ can be written by $\widehat{\bigotimes}_{1\leq i\leq s}U(\widehat{\mathfrak{gl}}(q_i))^{\otimes l_i}$ and $\delta_{C,D}$, where $C,D=i,j,p,q$. Since $\Phi_s$ is a homomorphism defined by replacing $\widehat{W}^{(r)}_{i,j}$ in $\Phi$ with $\widetilde{W}^{(r)}_{i,j}$, the well-definedness of $\widetilde{\mu}\circ\Phi$ induces the one of $\widetilde{\mu}\circ\Phi_s$. 

Next, we show the general case. 
Since we have already shown the case that $z=s$, we have a homomorphism
\begin{align*}
\Phi^{q_1,\cdots,q_z}_z\colon Y_{\hbar,\ve_z}(\widehat{\mathfrak{sl}}(q_z))\to\mathcal{U}(\mathcal{W}^{k+\sum_{v=z+1}^s l_vq_v}(\mathfrak{gl}(\sum_{i=1}^zl_iq_i),f_{q_1,q_2,\cdots,q_z}^{l_1,l_2\cdots,l_z}))
\end{align*}
given by the same formula as $\Phi^{q_1,\cdots,q_s}_s$ under the assumption that $\dfrac{\ve_z+q_z\hbar}{\hbar}=k+N$.
Then, we obtain a homomorphism
\begin{equation*}
\Phi^{q_1,\cdots,q_z}_z\circ\Psi^{q_z-q_{z+1},q_z}\colon Y_{\hbar,\ve_z}(\widehat{\mathfrak{sl}}(q_z-q_{z+1}))\to\mathcal{U}(\mathcal{W}^{k+\sum_{v=z+1}^s l_vq_v}(\mathfrak{gl}(\sum_{i=1}^zl_iq_i),f_{q_1,q_2,\cdots,q_z}^{l_1,l_2\cdots,l_z})).
\end{equation*}
Let us set $\sigma$ be a natural embedding from $\widehat{\bigotimes}_{i=1}^{z}U(\widehat{\mathfrak{gl}}(q_i))^{\otimes l_i}$ to $\widehat{\bigotimes}_{i=1}^{s}U(\widehat{\mathfrak{gl}}(q_i))^{\otimes l_i}$. Since $\widetilde{\mu}\circ\Phi_z=\sigma\circ\Phi^{q_1,\cdots,q_z}_z\circ\Psi^{q_z-q_{z+1},q_z}$ holds, the well-definedness of $\Phi_z$ is shown by the injectivity of the Miura-map.
\end{proof}
\section{Finite Yangians and non-rectangular finite $W$-algebras of type $A$}
The affine Yangian $Y_{\hbar,\ve}(\widehat{\mathfrak{sl}}(n))$ has a subalgebra $Y_{\hbar}(\mathfrak{sl}(n))$, which is a deformation of the universal enveloping algebra of the current algebra associated with $\mathfrak{sl}(n)$.
\begin{Definition}
Suppose that $n\geq3$. The finite Yangian $Y_{\hbar}(\mathfrak{sl}(n))$ is an associative algebra generated by 
\begin{equation*}
\{X^\pm_{i,r},H_{i,r}\mid 1\leq i\leq n-1,r\in\mathbb{Z}_{\geq0}\}
\end{equation*}
subject to the relations \eqref{Eq2.1}-\eqref{Eq2.5}, \eqref{Eq2.8} and \eqref{Eq2.10}.
\end{Definition}
We define a natural homomorphism 
\begin{equation*}
\kappa\colon Y_{\hbar}(\mathfrak{sl}(n))\to Y_{\hbar,\ve}(\widehat{\mathfrak{sl}}(n)),\ A_{i,r}\mapsto A_{i,r}
\end{equation*}
for $A=H,X^\pm$. By using $\Phi^{q_1,\cdots,q_u}_u$, we can construct a homomorphism from the finite Yangian $Y_{\hbar}(\mathfrak{sl}(q_u-q_{u+1}))$ to the finite $W$-algebra of type $A$ (see Section 6 in \cite{U4}). 

We set a degree on the $W$-algebra  $\mathcal{W}^k(\mathfrak{gl}(N),f_{q_1,q_2,\cdots,q_s}^{l_1,l_2\cdots,l_s})$ by $\text{deg}(W^{(r)}_{i,j}t^s)=s-r+1$. By Theorem~A.2.11 in \cite{NT}, the finite $W$-algebra $\mathcal{W}^{\text{fin}}(\mathfrak{gl}(N),f_{q_1,q_2,\cdots,q_s}^{l_1,l_2\cdots,l_s})$ can be defined as $\mathcal{U}_0/\sum_{r>0}\limits\mathcal{U}_{-r}\mathcal{U}_r$, where $\mathcal{U}_d$ is the set of degree $d$ elements of $W^k(\mathfrak{gl}(N),f_{q_1,q_2,\cdots,q_s}^{l_1,l_2\cdots,l_s})$. 
Let $\tilde{p}$ be a natural projection from $\mathcal{U}_0$ to $\mathcal{W}^{\text{fin}}(\mathfrak{gl}(N),f_{q_1,q_2,\cdots,q_s}^{l_1,l_2\cdots,l_s})$.

Since the image of $\Phi^{q_1,\cdots,q_u}_u$ is contained in $\mathcal{U}_0$, we obtain the homomorphism
\begin{equation*}
\Phi^{q_1,\cdots,q_u}_{f,u}\colon Y_{\hbar}(\mathfrak{sl}(q_u-q_{u+1}))\to \mathcal{W}^{\text{fin}}(\mathfrak{gl}(N),f_{q_1,q_2,\cdots,q_s}^{l_1,l_2\cdots,l_s}).
\end{equation*}
In \cite{BK}, Brundan-Kleshchev constructed a homomorphism from a shifted Yangian to a finite $W$-algebra of type $A$. By restricting this homomorphism, we can obtain a homomorphism from the Yangian associated with $\mathfrak{gl}(n)$ to the finite $W$-algebra. Let us recall the definition of the Yangian associated with $\mathfrak{gl}(n)$ (\cite{D1} and \cite{D2}).
\begin{Definition}
The Yangian $Y(\mathfrak{gl}(n))$ is an associative algebra whose generators are $\{t^{(r)}_{i,j}\mid r\geq0, 1\leq i,j\leq n\}$ and defining relations are
\begin{gather*}
t^{(0)}_{i,j}=\delta_{i,j},\\
[t^{(r+1)}_{i,j},t^{(s)}_{u,v}]-[t^{(r)}_{i,j},t^{(s+1)}_{u,v}]=t^{(r)}_{u,j}t^{(s)}_{i,v}-t^{(s)}_{u,j}t^{(r)}_{i,v}.
\end{gather*}
\end{Definition}
In the case when $\hbar\neq0$, the finite Yangian $Y_\hbar(\mathfrak{sl}(n))$ can be embedded into the Yangian $Y(\mathfrak{gl}(n))$ by 
\begin{gather*}
\iota_\hbar(h_{i,0})=t^{(1)}_{i,i}-t^{(1)}_{i+1,i+1},\quad \iota_\hbar(x^+_{i,0})=t^{(1)}_{i,i+1},\quad \iota_\hbar(x^-_{i,0})=t^{(1)}_{i+1,i},
\end{gather*}
\begin{align*}
\iota_\hbar(h_{i,1})&=
-\hbar t^{(2)}_{i,i}+\hbar t^{(2)}_{i+1,i+1}-\dfrac{i}{2}\hbar(t^{(1)}_{i,i}-t^{(1)}_{i+1,i+1}) -\hbar t^{(1)}_{i,i}t^{(1)}_{i+1,i+1} \\
&\quad+ \hbar\displaystyle\sum_{u=1}^{i}\limits t^{(1)}_{i,u}t^{(1)}_{u,i}-\hbar\displaystyle\sum_{u=1}^{i}\limits t^{(1)}_{i+1,u}t^{(1)}_{u,i+1}.
\end{align*}

Brundan-Kleshchev \cite{BK} constructed a homomorphism from the shifted Yangian to the $W$-algebra $\mathcal{W}^{\text{fin}}(\mathfrak{gl}(N),f_{q_1,q_2,\cdots,q_s}^{l_1,l_2\cdots,l_s})$. For $1\leq u\leq s$, the shifted Yangian contains the finite Yangian $Y(\mathfrak{gl}(q_u-q_{u+1}))$ by mapping $D^{(u)}_{i,j,r}$ in \cite{BK} to $t^{(r)}_{i,j}$.
By restricting the homomorphism to $Y(\mathfrak{gl}(q_u-q_{u+1}))$, we obtain the following homomorphism
\begin{equation*}
\widetilde{\Phi}_u\colon Y(\mathfrak{gl}(q_u-q_{u+1}))\to\mathcal{W}^{\text{fin}}(\mathfrak{gl}(N),f_{q_1,q_2,\cdots,q_s}^{l_1,l_2\cdots,l_s})
\end{equation*}
satisfying that
\begin{equation*}
\widetilde{\Phi}_u(t^{(1)}_{i,j})=\tilde{p}(\widetilde{W}^{(1)}_{i,j}),\ \widetilde{\Phi}_u(t^{(2)}_{i,j})=\tilde{p}(\widetilde{W}^{(2)}_{i,j}t)
\end{equation*}
for $1\leq i,j\leq q_u-q_{u+1}$.
By a direct computation, we obtain the following theorem.
\begin{Theorem}\label{cores}
The following relation holds:
\begin{equation*}
\widetilde{p}\circ\Phi_u\circ\kappa=\widetilde{\Phi}_u\circ\iota_{-1}.
\end{equation*}
\end{Theorem}
By Theorem~\ref{cores}, $\Phi$ is an affine analogue of the restriction of the homomorphism in \cite{BK}.

\section{The embedding of the finite Yangian of type $A$}
By the definition of $Y(\mathfrak{gl}(n))$, we have an embedding
\begin{equation*}
\widetilde{\Psi}^{n,m}_f\colon Y(\mathfrak{gl}(n))\to Y(\mathfrak{gl}(m)),\ t^{(r)}_{i,j}\mapsto t^{(r)}_{i,j}
\end{equation*}
for $m>n$.
By the definition of $\iota_\hbar$ and $\widetilde{\Psi}^{n,n+1}_f$, we find an embedding
\begin{equation*}
\Psi_f^{n,m}\colon Y_\hbar(\mathfrak{sl}(n))\to Y_\hbar(\mathfrak{sl}(m))
\end{equation*}
given by $A_{i,r}\mapsto A_{i,r}$ for $A=H,X^\pm$.
Let us consider the case that $s=1$ and $l_1=l,q_1=n$.
We denote by $\Phi^{n,l}$ the homomorphism $\Phi_1^{n}$ in the case that $s=1$, $l_1=l$ and $q_1=n$. We also denote this $W$-algebra by $\mathcal{W}^{k}(\mathfrak{gl}(ln),(l^{n}))$.
In the same way as Section 5 of \cite{U8}, we can give the embedding
\begin{equation*}
\iota\colon\mathcal{W}^{k}(\mathfrak{gl}(2n),(2^{n}))\to\mathcal{W}^{k}(\mathfrak{gl}(2(n+1)),(2^{n+1})),\ \widetilde{W}^{(r)}_{i,j}\mapsto \widetilde{W}^{(r)}_{i,j}.
\end{equation*}
In the same way as Theorem 5.5 of \cite{U8}, we can show the relation
\begin{equation*}
\Phi^{n+1,2}\circ\Psi^{n,n+1}=\iota\circ\Phi^{n,2}.
\end{equation*} 
Since $\iota$ is compatible with the degree of the $W$-algebra, $\iota$ induces the embedding
\begin{equation*}
\bar{\iota}\colon \mathcal{W}^{\text{fin}}(\mathfrak{gl}(2n),(2^n))\to\mathcal{W}^{\text{fin}}(\mathfrak{gl}(2(n+1)),(2^{n+1})),\ \tilde{p}(\widetilde{W}^{(r)}_{i,j}t^{r-1})\mapsto\tilde{p}(\widetilde{W}^{(r)}_{i,j}t^{r-1}).
\end{equation*}
Let us define a homomorphism $\Phi^{n,l}$ as $\widetilde{p}\circ\Phi^{n,l}\circ \kappa$. By the definition of $\Phi^{n,l}_f$, we find that
\begin{equation*}
\bar{\iota}\circ\Phi^{n,2}_f=\Phi^{n+1,2}_f\circ\Psi^{n,n+1}_f.
\end{equation*}
Thus, $\Psi^{n,m}$ can be regarded as an affine analogue of $\Psi^{n,m}_f$.

\section{Discussion of the new definition of the shifted affine Yangian}
In \cite{U11}, we constructed another homomorphism from the affine Yangian associated with $\widehat{\mathfrak{sl}}(m-n)$ to the one associated with $\widehat{\mathfrak{sl}}(m)$.
\begin{Theorem}
For positive integers $m,n$ satisfying that $m-n\geq3$, there exists a homomorphism
\begin{equation*}
\widetilde{\Psi}^{m-n,m}\colon Y_{\hbar,\ve+n\hbar}(\widehat{\mathfrak{sl}}(m-n))\to \widetilde{Y}_{\hbar,\ve}(\widehat{\mathfrak{sl}}(m))
\end{equation*}
given by
\begin{gather*}
\widetilde{\Psi}^{n,m}(X^+_{i,0})=\begin{cases}
E_{m,m-n+1}t&\text{ if }i=0,\\
E_{n+i,n+i+1}&\text{ if }i\neq 0,
\end{cases}\ 
\widetilde{\Psi}^{m-n,m}(X^-_{i,0})=\begin{cases}
E_{n+1,m}t^{-1}&\text{ if }i=0,\\
E_{n+i+1,n+i}&\text{ if }i\neq 0,
\end{cases}
\end{gather*}
and
\begin{align*}
\widetilde{\Psi}^{m-n,m}(H_{i,1})&= H_{i+n,1}+\hbar\displaystyle\sum_{s \geq 0}\limits\sum_{u=1}^n E_{u,n+i}t^{-s-1}E_{n+i,u}t^{s+1} \\
&\quad-\hbar\displaystyle\sum_{s \geq 0}\limits\sum_{u=1}^n E_{u,n+i+1}t^{-s-1} E_{n+i+1,u}t^{s+1},\\
\widetilde{\Psi}^{m-n,m}(X^+_{i,1})&=X^+_{i+n,1}+\hbar\displaystyle\sum_{s \geq 0}\limits\sum_{u=1}^n E_{u,n+i+1}t^{-s-1}E_{n+i,u}t^{s+1},\\
\widetilde{\Psi}^{m-n,m}(X^-_{i,1})&= X^-_{i+n,1}+\hbar\displaystyle\sum_{s \geq 0}\limits\sum_{u=1}^n E_{u,n+i}t^{-s-1}E_{n+i+1,u}t^{s+1}
\end{align*}
for $i\neq0$.
\end{Theorem}

We denote the set of simple roots of $\widehat{\mathfrak{sl}}(m)$ by $\{\alpha_i\}_{i=0}^{m-1}$ and its dual basis by $\{\alpha^*_i\}_{i=0}^{m-1}$. 
The shifted affine Yangian associated with $\widehat{\mathfrak{sl}}(m)$ and its anti-dominant weight $\beta$ is defined as a subalgebra of $Y_{\hbar,\ve}(\widehat{\mathfrak{sl}}(m))$ generated by
\begin{equation*}
\{X^\pm_{i,r^\pm_i},H_{i,r_i}\mid 0\leq i\leq m-1,r^+_i\geq0,r_i,r^-_i\geq-\beta(\alpha_i)\}.
\end{equation*}
Creutzig-Diaconescu-Ma conjectured that there exists a surjective homomorphism from the shifted affine Yangian to the universal enveloping algebra of a non-rectangular $W$-algebra of type $A$ if we change the definition of the shifted affine Yangian. Let us consider the case that $s=2$, $l_2=1$, $q_1=m$ and $q_2=n$. In this case, based on the form of $W^{(r)}_{i,j}$ and Theorem~\ref{May}, it is naturally expected that we can extend homomorphisms $\Phi_1$ and $\Phi_2$ directly to the subalgebra of $Y_{\hbar,\ve}(\widehat{\mathfrak{sl}}(q_1))$ generated by $\Psi^{q_1-q_2,q_1}(Y_{\hbar,\ve}(\widehat{\mathfrak{sl}}(q_1-q_2)))$, $\widetilde{\Psi}^{q_2,q_1}(Y_{\hbar,\ve+(q_1-q_2)\hbar}(\widehat{\mathfrak{sl}}(q_2)))$, $X^+_{m-n,0}$ and $\{E_{m-n,m-n}t^{y}-E_{m-n+1,m-n+1}t^y\mid y\in\mathbb{Z}\}$ and this will lead to the new definition of the shifted affine Yangian. By the definition of $\Psi^{m-n,m}$ and $\widetilde{\Psi}^{n,m}$, we have
\begin{align}
&\quad[\Psi^{m-n,m}(\widetilde{H}_{i,1}),X^+_{m-n,0}]\nonumber\\
&=-\delta_{i,m-n-1}X^+_{m-n,1}+\hbar\displaystyle\sum_{y \geq 0} \limits E_{i,m-n+1}t^{-y-1} E_{m-n,i}t^{y+1}\\
&\quad-\hbar\displaystyle\sum_{y \geq 0} \limits E_{i+1,m-n+1}t^{-y-1} E_{m-n,i+1}t^{y+1}+\delta_{i,m-n-1}\hbar\displaystyle\sum_{y \geq 0}\limits\sum_{u=m-n+1}^m E_{m-n,u}t^{-y-1} E_{u,m-n+1}t^{y+1}\label{Ee1}
\end{align}
and
\begin{align}
&\quad[\widetilde{\Psi}^{n,m}(\widetilde{H}_{i,1}),X^-_{m-n,0}]\nonumber\\
&=-\delta_{i,1}X^+_{m-n,1}+\hbar\displaystyle\sum_{y \geq 0}\limits E_{m-n,m-n+i}t^{-y-1}E_{m-n+i,m-n+1}t^{y+1}\nonumber\\
&\quad-\delta_{i,1}\hbar\displaystyle\sum_{y \geq 0}\limits\sum_{u=1}^{m-n} E_{u,m-n+1}t^{-y-1}E_{m-n,u}t^{y+1}-\hbar\displaystyle\sum_{y \geq 0}\limits E_{m-n,m-n+i+1}t^{-y-1}E_{m-n+i+1,m-n+1}t^{y+1}.\label{Ee2}
\end{align}
Thus, we find that $X^+_{m-n,1}$ is contained in this subalgebra.

The form of this extension is expected as follows:
\begin{gather*}
\Phi(X^{+,(r)}_{i,0})=\begin{cases}
\widetilde{W}^{(1),r}_{m+1-i,m-i}&\text{ if }i\neq0,\\
\widetilde{W}^{(1),r}_{m-n,m}t&\text{ if }i=0,
\end{cases}
\Phi(X^{-,(r)}_{i,0})=\begin{cases}
\widetilde{W}^{(1),r}_{m-i,m+1-i}&\text{ if }i\neq0,\\
\widetilde{W}^{(1),r}_{m,m-n}t^{-1}&\text{ if }i=0,
\end{cases}\\
\Phi(X^+_{m-n,0})=W^{(1)}_{m-n,m-n+1},\\
\Phi(E_{m-n,m-n}t^{y}-E_{m-n+1,m-n+1}t^y)=W^{(1)}_{m-n,m-n}t^y-W^{(1)}_{m-n+1,m-n+1}t^y
\end{gather*}
and
\begin{align*}
\Phi(\widetilde{H}^{(r)}_{i,1})&=-\hbar(\widetilde{W}^{(2),r}_{i,i}t-\widetilde{W}^{(2),r}_{i+1,i+1}t)-\dfrac{i}{2}\hbar\Phi(H_{i,0})-\dfrac{\hbar}{2}(\widetilde{W}^{(1),r}_{i,i})^2-\dfrac{\hbar}{2}(\widetilde{W}^{(1),r}_{i+1,i+1})^2\\
&\quad+\hbar\displaystyle\sum_{y \geq 0}  \limits\displaystyle\sum_{u=1}^{i}\limits \widetilde{W}^{(1),r}_{i,u}t^{-y}\widetilde{W}^{(1),r}_{u,i}t^y+\hbar\displaystyle\sum_{s \geq 0} \limits\displaystyle\sum_{u=i+1}^{m-n}\limits \widetilde{W}^{(1),r}_{i,u}t^{-y-1} \widetilde{W}^{(1),r}_{u,i}t^{y+1}\\
&\quad-\hbar\displaystyle\sum_{s \geq 0}\limits\displaystyle\sum_{u=1}^{i}\limits \widetilde{W}^{(1),r}_{i+1,u}t^{-y} \widetilde{W}^{(1),r}_{u,i+1}t^y-\hbar\displaystyle\sum_{s \geq 0}\limits\displaystyle\sum_{u=i+1}^{m-n} \limits \widetilde{W}^{(1),r}_{i+1,u}t^{-y-1} \widetilde{W}^{(1),r}_{u,i+1}t^{y+1}\text{ for }r=1,2,
\end{align*}
where we denote $\Psi^{m-n,m}(A_{i,r})$ and $\widetilde{\Psi}^{n,m}(A_{i,r})$ by $A^{(1)}_{i,r}$ and $A^{(2)}_{i,r}$ for $A=X^\pm,H$. 
By Theorem~\ref{Tho1}, we obtain
\begin{align*}
&\quad[\Phi(\widetilde{H}^{(1)}_{i,1}),\Phi(X^+_{m-n,0})]\\
&=\delta_{i,m-n-1}\hbar W^{(2)}_{m-n,m-n+1}t+\hbar(W^{(1)}_{m-n,i})_{(-1)}W^{(1)}_{i,m-n+1}t-\hbar(W^{(1)}_{m-n,i+1})_{(-1)}W^{(1)}_{i+1,m-n+1}t\\
&\quad+\delta_{i,m-n-1}\dfrac{m-n-1}{2}\hbar W^{(1)}_{m-n,m-n+1}-\delta_{i,m+n-1}\dfrac{\hbar}{2}\{W^{(1)}_{m-n,m-n},W^{(1)}_{m-n,m-n+1}\}\\
&\quad+\hbar\displaystyle\sum_{y \geq 0} \limits W^{(1)}_{i,m-n+1}t^{-y-1} W^{(1)}_{m-n,i}t^{y+1}\\
&\quad-\delta_{i,m-n-1}\hbar\displaystyle\sum_{y \geq 0}\limits\displaystyle\sum_{u=1}^{i} \limits W^{(1)}_{m-n,u}t^{-y} W^{(1)}_{u,m-n+1}t^{y}\\
&\quad-\delta_{i,m-n-1}\hbar\displaystyle\sum_{y \geq 0}\limits\displaystyle\sum_{u=i+1}^{m-n} \limits W^{(1)}_{m-n,u}t^{-y-1} W^{(1)}_{u,m-n+1}t^{y+1}\\
&\quad-\hbar\displaystyle\sum_{y \geq 0}\limits W^{(1)}_{i+1,m-n+1}t^{-y-1} W^{(1)}_{m-n,i+1}t^{y+1}.
\end{align*}
By the definition of the universal enveloping algebra, we have
\begin{align*}
&\quad(W^{(1)}_{m-n,i})_{(-1)}W^{(1)}_{i,m-n+1}t\\
&=\sum_{y\geq0}\limits (W^{(1)}_{m-n,i}t^{-y-1}W^{(1)}_{i,m-n+1}t^{y+1}+W^{(1)}_{i,m-n+1}t^{-y}W^{(1)}_{m-n,i}t^y),\\
&\quad(W^{(1)}_{m-n,i+1})_{(-1)}W^{(1)}_{i+1,m-n+1}t\\
&=\sum_{y\geq0}\limits (W^{(1)}_{m-n,i+1}t^{-y-1}W^{(1)}_{i+1,m-n+1}t^{y+1}+W^{(1)}_{i+1,m-n+1}t^{-y}W^{(1)}_{m-n,i+1}t^y),
\end{align*}
Thus, we obtain
\begin{align*}
&\quad[\Phi(\widetilde{H}^{(1)}_{i,1}),\Phi(X^+_{m-n,0})]\\
&=\delta_{i,m-n-1}\hbar W^{(2)}_{m-n,m-n+1}t+\delta_{i,m-n-1}\dfrac{m-n-1}{2}\hbar W^{(1)}_{m-n,m-n+1}\\
&\quad+\hbar W^{(1)}_{m-n,i}W^{(1)}_{i,m-n+1}+\hbar \sum_{y\geq0}\limits W^{(1)}_{i,m-n+1}t^{-y-1}W^{(1)}_{m-n,i}t^{y+1}\\
&\quad-\hbar W^{(1)}_{m-n,i+1}W^{(1)}_{i+1,m-n+1}-\hbar \sum_{y\geq0}\limits W^{(1)}_{i+1,m-n+1}t^{-y-1}W^{(1)}_{m-n,i+1}t^{y+1}\\
&\quad-\delta_{i,m+n-1}\dfrac{\hbar}{2}\{W^{(1)}_{m-n,m-n},W^{(1)}_{m-n,m-n+1}\}\\
&\quad-\delta_{i,m-n-1}\hbar\displaystyle\sum_{y \geq 0}\limits\displaystyle\sum_{u=1}^{i} \limits W^{(1)}_{m-n,u}t^{-y} W^{(1)}_{u,m-n+1}t^{y}\\
&\quad-\delta_{i,m-n-1}\hbar\displaystyle\sum_{s \geq 0}\limits\displaystyle\sum_{u=i+1}^{m-n} \limits W^{(1)}_{m-n,u}t^{-y-1} W^{(1)}_{u,m-n+1}t^{y+1}.
\end{align*}
Similarly, we can obtain
\begin{align*}
&\quad[\Phi(\widetilde{H}^{(2)}_{i,1}),\Phi(X^-_{m-n,0})]\\
&=\delta_{i,1}\hbar W^{(2)}_{m-n,m-n+1}t+\sum_{y\geq0}\hbar(W^{(1)}_{m-n,m-n+i})t^{-y-1}W^{(1)}_{m-n+i,m-n+1}t^{y+1}\\
&\quad-\sum_{y\geq0}\hbar(W^{(1)}_{m-n,m-n+i+1})t^{-y-1}W^{(1)}_{m-n+i+1,m-n+1}t^{y+1}\\
&\quad+\delta_{i,1}\dfrac{1}{2}\hbar W^{(1)}_{m-n,m-n+1}-\delta_{i,1}\dfrac{\hbar}{2}\{\widetilde{W}^{(2),2}_{1,1},W^{(1)}_{m-n,m-n+1}\}\\
&\quad-\delta_{i,1}\hbar\displaystyle\sum_{y \geq 0}  \limits\displaystyle\sum_{u=1}^{i}\limits W^{(1)}_{m-n,m-n+u}t^{-y}W^{(1)}_{m-n+u,m-n+1}t^y\\
&\quad-\delta_{i,1}\hbar\displaystyle\sum_{y \geq 0}  \limits\displaystyle\sum_{u=i+1}^{m-n}\limits W^{(1)}_{m-n,m-n+u}t^{-y}W^{(1)}_{m-n+u,m-n+1}t^y.
\end{align*}
Thus, we cannot determine $\Phi(X^+_{m-n,1})$, which is compatible with the relations \eqref{Ee1} and \eqref{Ee2}. This means that we cannot extend the homomorphism $\Phi_1$ and $\Phi_2$ directly to the subalgebra of $Y_{\hbar,\ve}(\widehat{\mathfrak{sl}}(m))$.
We expect that these results of computation will lead to the new definition of the shifted affine Yangian.
\appendix
\section{Relationship between edge contractions for the quantum affine algebra and the finite Yangian}
Let $q$ be a non-zero complex number satisfying that $q^a\neq 1$ for any $a\in\mathbb{Z}\setminus\{0\}$.
\begin{Definition}\label{ho}
Let $n\geq 3$. The quantum affine algebra $U_q(\widehat{\mathfrak{sl}}(n))$ is an associative algebra generated by 
\begin{equation*}
\{e_i,f_i,t^{\pm1}_i\mid0\leq i\leq n-1\}
\end{equation*}
with the commutator relations:
\begin{gather*}
t_it_j=t_jt_i,\ t_it^{-1}_i=t^{-1}_it_i=1\text{ if }0\leq i\leq n-1,\\
t_ie_j=q^{a_{i,j}}e_jt_i,\ t_if_j=q^{-a_{i,j}}f_jt_i\text{ if }0\leq i,j\leq n-1,\\
[e_i,f_j]=\delta_{i,j}\dfrac{t_i-t^{-1}_i}{q-q^{-1}}\text{ if }0\leq i,j\leq n-1,\\
[e_i,e_j]=0,\ [f_i,f_j]=0\text{ if }a_{i,j}=0,\\
[e_i,[e_i,e_j]_{q^{-1}}]_q=0,\ [f_i,[f_i,f_j]_{q^{-1}}]_q=0\text{ if }a_{i,j}=-1,
\end{gather*}
where 
\begin{gather*}
[a]_q=\dfrac{q^a-q^{-a}}{q-q^{-1}}\text{ for }a\in\mathbb{Z},\ 
[x,y]_{q^a}=xy-q^ayx
\end{gather*}
\end{Definition}
By the definition of $U_q(\widehat{\mathfrak{sl}}(n))$, we have an isomorphism 
\begin{equation*}
Ro\colon U_q(\widehat{\mathfrak{sl}}(n))\to U_q(\widehat{\mathfrak{sl}}(n))
\end{equation*}
given by
\begin{equation*}
e_i\mapsto e_{n-i-1},\ f_i\mapsto f_{n-i-1},\ t_i\mapsto t_{n-i-1}.
\end{equation*}

Li \cite{Li} constructed an embedding from the quantum affine algebra $U_q(\widehat{\mathfrak{sl}}(n))$ to  the quantum affine algebra $U_q(\widehat{\mathfrak{sl}}(n+1))$.
\begin{Theorem}[Theorem 2.1.1 in \cite{Li}]
Let $\ve=\pm1$ and $0\leq r\leq n-1$. There exists an embedding
\begin{equation*}
\phi\colon U_q(\widehat{\mathfrak{sl}}(n))\to U_q(\widehat{\mathfrak{sl}}(n+1))
\end{equation*}
given by
\begin{gather*}
e_i\mapsto \begin{cases}
e_i&\text{ if }0\leq i\leq r-1,\\
[e_r,e_{r+1}]_{q^\ve}&\text{ if }i=r,\\
e_{i+1}&\text{ if }r+1\leq i\leq n-1,
\end{cases}\\
f_i\mapsto \begin{cases}
f_i&\text{ if }0\leq i\leq r-1,\\
[f_{r+1},f_r]_{q^{-\ve}}&\text{ if }i=r,\\
f_{i+1}&\text{ if }r+1\leq i\leq n-1,
\end{cases}
t_i\mapsto\begin{cases}
t_i&\text{ if }0\leq i\leq r-1,\\
t_rt_{r+1}&\text{ if }i=r,\\
t_{i+1}&\text{ if }r+1\leq i\leq n-1.
\end{cases}
\end{gather*}
\end{Theorem}
There exists another presentation of the quantum affine algebra $U_q(\widehat{\mathfrak{sl}}(n))$ given in Theorem 4.7, \cite{Beck}. Let us consider the homomorphism $\widetilde{\phi}=Ro^{-1}\circ\phi\circ Ro$.
By the definition of $\widetilde{\phi}$, we can rewrite
\begin{gather*}
\widetilde{\phi}(e_i)=\begin{cases}
e_i&\text{ if }0\leq i\leq n-r-2,\\
[e_{i+1},e_{i}]_{q^\ve}&\text{ if }i=n-r-1,\\
e_{i+1}&\text{ if }n-r\leq i\leq n-1,
\end{cases}\\
\widetilde{\phi}(f_i)=\begin{cases}
f_i&\text{ if }0\leq i\leq n-r-2,\\
[f_{i},f_{i+1}]_{q^{-\ve}}&\text{ if }i=n-r-1,\\
f_{i+1}&\text{ if }n-r\leq i\leq n-1,
\end{cases}
\widetilde{\phi}(t_i)=\begin{cases}
t_i&\text{ if }0\leq i\leq n-r-2,\\
t_{i+1}t_{i}&\text{ if }i=n-r-1,\\
t_{i+1}&\text{ if }n-r\leq i\leq n-1.
\end{cases}
\end{gather*}
Let us consider the case that $\ve=-1$ and $r=0$. By using the presentation of Theorem~4.7 in \cite{BK}, we can rewrite $\widetilde{\phi}$ as
\begin{equation*}
x^\pm_{i,0}\mapsto x^\pm_{i,0},\ k_i\mapsto k_i,\ q^c\mapsto q^c
\end{equation*}
and
\begin{align*}
h_{1,1}&\mapsto-[\cdots[[e_{0},e_{n}]_{q^{-1}},e_{n-1}]_{q^{-1}}\cdots,e_{2}]_{q^{-1}},e_1]_{q^{-2}}=h_{1,1},\\
h_{1,-1}&\mapsto-[f_1,[f_{2},\cdots,[f_{n-1},[f_{n},f_{0}]_q]_q\cdots]_q]_{q^2}=h_{1,-1}.
\end{align*}
Since $U_q(\widehat{\mathfrak{sl}}(n))$ is generated by $x^\pm_{i,0},k_i,q^c,h_{1,\pm1}$, we find that $\psi(A_{i,r})=A_{i,r}$ for $A=h,x^\pm$.
Thus, $\widetilde{\phi}$ is corresponding to $\Psi_f^{n,n+1}$.

Li \cite{Li3} showed that the embedding $\phi$ is compatible with the Schur-Weyl dulaity. In the last of this section, we give the relationship with the homomorphism $\Psi_f^{n,n+1}$ and the Schur-Weyl duality.
Drinfeld \cite{D3} and Arakawa \cite{A3} showed the Schur-Weyl duality for the finite Yangian of type $A$ and the degenerate affine Hecke algebra.
For a right degenerate affine Hecke algebra module $M$, let $\Gamma_n\colon Y(\mathfrak{gl}(n))\to\End(M\otimes_{\mathbb{C}[S_l]}(\mathbb{C}^n)^{\otimes l})$ be a homomorphism determined by the Schur-Weyl duality. We note that $\Gamma_n(t_{i,j}^{(r)})$ can be considered a endomorphism of $M\otimes_{\mathbb{C}[S_l]}(\mathbb{C}^{n+1})^{\otimes l}$. Then, we find the relation
\begin{equation*}
\Gamma_{n+1}\circ\Psi_f^{n,n+1}=\Gamma_n.
\end{equation*}

\section*{Data Availability}
The authors confirm that the data supporting the findings of this study are available within the article and its supplementary materials.
\section*{Declarations}
\subsection*{Funding}
This work was supported by Japan Society for the Promotion of Science Overseas Research Fellowships, Grant Number JP2360303. 
\subsection*{Conflicts of interests/Competing interests}
The authors have no competing interests to declare that are relevant to the content of this article.
\subsection*{Declaration of generative AI in scientific writing}
Authors must not list or cite AI and AI-assisted technologies as an author or co-author on the manuscript since authorship implies responsibilities and tasks that can only be attributed to and performed by humans.

\end{document}